\documentclass{amsart}
\pdfoutput=1
\usepackage{amssymb,latexsym,amsmath,amscd,graphicx,graphics,epic,eepic,enumerate, amsfonts, amssymb, mathrsfs, pifont}

\theoremstyle{plain}
\newtheorem{theorem}{Theorem}[section]
\newtheorem{lemma}[theorem]{Lemma}
\newtheorem{proposition}[theorem]{Proposition}
\newtheorem{corollary}[theorem]{Corollary}

\theoremstyle{definition}

\theoremstyle{remark}
\newtheorem{remark}[theorem]{Remark}

\def \R{\mathbb{R}}
\def \Z{\mathbb{Z}}
\def \C{\mathbb{C}}
\def \t {\mathfrak{t}}
\def \M {\mathfrak{M}} 
\def \s {\mathfrak{s}}
\def \S{\Sigma}

\begin{document}

\title{On the relative slice Thurston-Bennequin inequality}

\author{Georgi D. Gospodinov}

\address{ Georgi D. Gospodinov, Bard College, PO Box 5000, Annandale-on-Hudson, NY 12504-5000}

\email{ggospodi@bard.edu}


\begin{abstract} 
We derive a relative version of the slicing Bennequin inequalities for cobordant Legendrian knots, and review a few proofs of the result.
\end{abstract}

\maketitle

\section{Introduction}

In \cite{georgi}, we studied relative invariants of Legendrian knots that are homologous to a fixed knot in a contact 3-manifold. For Legendrian homologous knots $K$ and $J$ in $(M,\xi)$ and an embedded surface $\Sigma$ with $K\cup J=\partial\Sigma$, the {\em Thurston-Bennequin invariant of $K$ relative to $J$} is $\widetilde{tb}_\Sigma(K,J):=tw_K(\xi,Fr_{\Sigma})-tw_{J}(\xi,Fr_{\Sigma}),$ where $Fr_\Sigma$ denotes the Seifert framing that $K$ (resp. $J$) inherits from $\Sigma$, and $tw(\xi,Fr_\Sigma)$ denotes the number of $2\pi$-twists (with sign) of the contact framing relative to $Fr_\Sigma$ along $K$ or $J$. For push-offs $K'$ and $J'$ of $K$ and $J$ in the direction normal to the contact planes, $\widetilde{tb}_\Sigma(K,J)=K'\cdot\Sigma-J'\cdot\Sigma=lk_\Sigma(K',K)-lk_\Sigma(J',J)$. Since $\xi\rvert_\Sigma$ is trivial, its restriction to $K$ gives a map $\sigma:\xi\rvert_K\rightarrow K\times\R^2$, under which a non-zero tangent vector field $v_K$ to $K$ traces out a path of vectors in $\R^2$. Similarly for $J$. The \textit{relative rotation number of $K$} is $\widetilde{r}_\Sigma(K,J):=w_\sigma(v_K)-w_\sigma(v_J)$. Equivalently, $\widetilde{r}_\Sigma(K,J)=e(\xi,v_K\cup v_J)([\Sigma])$.

The goal of this note is to extend this study to the case where the surface is in a 4-dimensional cobordism between contact 3-manifolds, each containing one of the knots cobounding the surface. This requires some 4-dimensional techniques and avoids the difficulty of keeping track of the surface under Legendrian isotopy of each knot, in particular, the two knots do not intersect (compare with \cite{georgi}). 

\section{Definitions}

We define the invariants in this case by extending the definition of invariants of slice Legendrian knots introduced by Mrowka and Rollin in \cite{MR}.

Let $W$ be an oriented 4-manifold with $\partial W=(M_1,\xi_1)\cup\cdots\cup(M_n,\xi_n)$, where $(M_i,\xi_i)$ is a (connected) closed contact 3-manifold  for $i=1,\dots,n$ with coorientable contact structure $\xi_i$ which provides an orientation of $(M_i,\xi_i)$.  Let $K_i\subset M_i$ be Legendrian and assume that there is an embedded surface $F$ in $W$ with $\partial F=K_1\cup\cdots\cup K_n$. Since the $\xi_i$ are coorientable, we can choose a vector field $v_i$ transverse to $\xi_i$ on $M_i\subset\partial W$  such that $v_i$ points along the positive normal of $\xi_i$. Note that we can extend the $v_i$ to a vector field on a neighborhood of $\partial W$ such that $v\rvert_{M_i}=v_i$, and further, we can extend $v$ to a neighborhood of $F$ and then to all of $W$, since the obstruction for this lies in $H^2(W,\partial W,\pi_2(S^3))$, which is trivial. Let $\{\phi_t\}$ be the flow of $v$, and for a small $\epsilon>0$, take a push-off $F'=\varphi_\epsilon(F)$ such that $K_i'=\varphi_\epsilon(K_i)$ and $\partial F'=K_1'\cup K_2'$. Since the boundaries of $F$ and $F'$ are disjoined, their intersection $F'\cdot F$ in $W$ is well-defined, and the {\em Thurston-Bennequin number} of the $K_i$ is defined as 
$$\widetilde{tb}(K_1,\dots,K_n,F)=F\cdot F'.$$ 
Note that $\widetilde{tb}(K_1,\dots,K_n,F)$ depends on $F$ only through the relative homology class $[F]\in H_2(W,K_1\cup\cdots\cup K_n;\Z)$ and is well-defined. Now let $\mathfrak{s}\in Spin^{\mathbb{C}}(W)$ and let $\mathfrak{t_{\xi_i}}$ be the canonical $Spin^{\mathbb{C}}$-structure on $M_i$ associated to $\xi_i$, $i=1,\dots,n$. We have isomorphisms $h_i:\mathfrak{s}\rvert_{M_i}\rightarrow \xi_i$ which induce isomorphisms $\det(h_i):\det(\mathfrak{s})\rvert_{M_i}\rightarrow\xi_i$ for a choice of complex structure on $\xi_i$. Let $u_{K_i}$ be the positive unit tangent vector fields along $K_i$. Define the {\em rotation number} of the Legendrian knots $K_i$ as 
$$\widetilde{r}(K_1,\dots,K_n,F,\mathfrak{s},h_1,\dots,h_n):=$$
$$c_1\Big{(}\det(\mathfrak{s}),\det(h_1)^{-1}(u_{K_1})\cup\cdots\cup \det(h_n)^{-1}(u_{K_n})\Big{)}([F]).$$ 
Again, $\widetilde{r}(K_1,\dots, K_n,F,\mathfrak{s},h_1,\dots,h_n)$ depends on $F$ only through its relative homology class $[F]\in H_2(W,K_1\cup\cdots\cup K_n;\Z)$ and on the isomorphism type of the pairs $(\mathfrak{s},h_i)$ in $Spin^{\mathbb{C}}(W,\xi_i)$. Consider the case when there exists a symplectic form $\omega$ on $W$ such that $\omega\rvert_{\xi_i}>0$ (so $W$ is a weak symplectic filling of $(M_1,\xi_1)\cup\cdots\cup (M_n,\xi_n)$). Then $\omega$ determines a canonical $Spin^\C$-structure $\s_\omega$ on $W$ and canonical isomorphisms $h^\omega_i:\s_\omega\rvert_{M_i}\rightarrow\mathfrak{t}_{\xi_i}$ for $i=1,\dots,n$ giving us the rotation number $\widetilde{r}(K_1,\dots,K_n,F,\s_\omega,h^{\omega}_1,\dots,h^{\omega}_n)$.

Of interest to us is the case when $n=2$. Let $(M_i,\xi_i)$ be a closed contact 3-manifold  for $i=1,2$ with coorientable contact structure $\xi_i$ and consider an oriented compact 4-manifold $W$ such that the oriented boundary of $W$ equals $-M_1\cup M_2$. Let $K_i$ be a Legendrian knot in $M_i$ and assume that there is an embedded surface $F$ in $W$ with $\partial F=K_1\cup K_2$. We call such Legendrian knots $K_1$ and $K_2$ {\em cobordant}. As above, we obtain a vector field $v$ on $\partial W$ transverse to $\xi_1$ and $\xi_2$ such that $v$ points along the negative normal of $\xi_1$ and the positive normal of $\xi_2$, and we extend $v$ to all of $W$. Following the above construction, we have the {\em relative Thurston-Bennequin number} of $K_1$ and $K_2$ as 
$$\widetilde{reltb}(K_1,K_2,F):=\widetilde{tb}(K_1,K_2,F)=F'\cdot F.$$
It depends on $F$ only through the relative homology class $[F]\in H_2(W,K_1\cup K_2;\Z)$ and is well-defined. Also, for $\mathfrak{s}\in Spin^{\mathbb{C}}(W)$ we define the {\em relative rotation number} of the Legendrian knots $K_1$ and $K_2$ to be 
$$\widetilde{relr}(K_1,K_2,F,\mathfrak{s},h_1,h_2):=\widetilde{r}(K_1,K_2,F,\mathfrak{s},h_1,h_2).$$
If there exists a symplectic form $\omega$ on $W$ such that $\omega\rvert_{\xi_1}<0$ and $\omega\rvert_{\xi_2}>0$ (so $(M_1,\xi_1)$ is the {\em concave} end and $(M_2,\xi_2)$ is the {\em convex} end of the {\em symplectic cobordism} $(W,\omega)$), then $\omega$ determines a canonical $Spin^\C$-structure $\s_\omega$ on $W$ and canonical isomorphisms $h_{\omega,i}:\s_\omega\rvert_{M_i}\rightarrow\mathfrak{t}_{\xi_i}$ for $i=1,2$ giving us the relative rotation number 
$$\widetilde{relr}(K_1,K_2,F,\s_\omega,h^{\omega}_1,h^{\omega}_2)=\widetilde{r}(K_1,K_2,F,\s_\omega,h^{\omega}_1,h^{\omega}_2).$$

\begin{remark}\label{remark1}(Basic properties) Since $K_i$ is oriented as a component of $\partial F$,
$$\widetilde{tb}(K_1,\dots,K_n,F)=\widetilde{tb}(-K_1,\dots,-K_n,-F),\ \ $$ 
and
$$\widetilde{r}(K_1,\dots,K_n,F,\s,h_1,\dots,h_n)=-\widetilde{r}(-K_1,\dots,-K_n,-F,\s,h_1,\dots,h_n).$$
Moreover, if there exists a Seifert surface $\S_i\hookrightarrow M_i$ with $\partial\S_i=K_i$, then we can define $tb_{\S_i}(K_i)$ and $r_{\S_i}(K_i)$ in the classical sense for $i=1,\dots,n$. Then 
$$\widetilde{tb}(K_1,\dots,K_n,F)=\sum_{i=1}^n tb_{\S_i}(K_i).$$
Similarly, 
$$\widetilde{r}(K_1,\dots,K_n,F,\s,h_1,\dots,h_n)=\sum_{i=1}^n r_{\S_i}(K_i).$$ 
In the relative case
$$\widetilde{reltb}(K_1,K_2,F)=tb_{\S_2}(K_2)-tb_{\S_1}(K_1)$$
and
$$\widetilde{relr}(K_1,K_2,F,\s_\omega,h_1,h_2)=r_{\S_2}(K_2)-r_{\S_1}(K_1),$$
which is a direct generalization of the relative invariants defined in \cite{georgi}. 
\end{remark}

\section{Statement of Results}

We generalize the construction outlined in \cite{MR, Wu1} to obtain the following.

\begin{theorem}\label{rel-slice-genus} 
Let $W$ be a 4-manifold with boundary $(M_1,\xi_1)\cup\cdots\cup (M_n,\xi_n)$, where $(M_i,\xi_i)$ are connected contact 3-manifolds, and let $K_i\subset (M_i,\xi_i)$ be a Legendrian knot for $i=1,\dots,n$. 
\begin{enumerate}[(a)]
\item Let $F$ be any embedded surface in $W$ with $\partial F=K_1\cup\cdots\cup K_n$ and $W'$ be any 4-manifold obtained from $W$ by attaching enough 1-handles away from $F$ so that $W'$ has connected boundary $(M,\xi)$ such that $\xi\rvert_{M_i}=\xi_i$ for all $i$. 

\subitem (i) If there is a $Spin^{\mathbb{C}}$-structure $\s$ on $W'$ with $$F^+_{W'\setminus B,\s\rvert_{W'\setminus B}}\big{(}c^+(\xi))\big{)}\neq0,$$ where $B$ is an embedded $4$-ball in the interior of $W'$, then there is an isomorphism $h:\s\rvert_{M}\rightarrow \t_{\xi}$ such that $$\widetilde{tb}(K_1,\dots,K_n,F)+\rvert \widetilde{r}(K_1,\dots,K_n,F,\s\rvert_{W'},h)\rvert\leq-\chi(F).$$

\subitem(ii) If there is a $Spin^{\mathbb{C}}$-structure $\s$ on $W'$ such that $$sw_{(W',\xi)}(\s,h)\neq 0$$ for an isomorphism $h:\s\rvert_M\rightarrow\t_\xi$, then $$\widetilde{tb}(K_1,\dots,K_n,F)+\rvert \widetilde{r}(K_1,\dots,K_n,F,\s\rvert_{W'},h)\rvert\leq-\chi(F).$$

\item Let $F$ be any embedded surface in $W$ with $\partial F=K_1\cup\cdots\cup K_n$. 

\subitem (i) If there is a $Spin^{\mathbb{C}}$-structure $\s$ on $W$ with $$\widehat{F}_{W\setminus B,\s\rvert_{W\setminus B}}\big{(}c(\xi_1)\otimes\cdots\otimes c(\xi_n)\big{)}\neq0,$$ where $B$ is an embedded $4$-ball in the interior of $W$, then there are isomorphisms $h_i:\s\rvert_{M_i}\rightarrow \t_{\xi_i}$ for $i=1,\dots, n$, such that $$\widetilde{tb}(K_1,\dots,K_n,F)+\rvert \widetilde{r}(K_1,\dots,K_n,F,\s\rvert_{W},h_1,\dots,h_n)\rvert\leq-\chi(F).$$

\subitem(ii) If there is a $Spin^{\mathbb{C}}$-structure $\s$ on $W$ with $\s\rvert_{M_i}=\xi_i$ such that $$sw_{(W,\xi_1\cup\cdots\cup\xi_n)}(\s,h_1,\dots,h_n)\neq 0$$ for isomorphisms $h_i:\s\rvert_{M_i}\rightarrow \t_{\xi_i}, i=1,\dots, n$, then $$\widetilde{tb}(K_1,\dots,K_n,F)+\rvert \widetilde{r}(K_1,\dots,K_n,F,\s\rvert_{W},h_1,\dots,h_n)\rvert\leq-\chi(F).$$

\item If $W$ is a weak symplectic filling of $(M_1,\xi_1)\cup\cdots\cup (M_n,\xi_n)$, then for any embedded surface $F$ with $\partial F=K_1\cup\cdots\cup K_n$, 
$$\widetilde{tb}(K_1,\dots,K_n,F)+\rvert \widetilde{r}(K_1,\dots,K_n,F,\s_\omega,h^{\omega}_1,\dots,h^{\omega}_n)\rvert\leq-\chi(F).$$
\end{enumerate}
\end{theorem}

For part (a) of Theorem \ref{rel-slice-genus} above, we ensure that the cobordisms have connected boundary so that the $F^+$ homomorphisms are defined and the standard approach of Theorem \ref{slice-genus} applies: cap off $F$ with the cores of Weinstein 2-handles attached along the $K_i$ and apply the adjunction inequality (Theorem \ref{adjunction}). Notice that adding 1-handles gives potentially different diffeomorphism types of our cobordisms, however, this does not affect the genus bound of the surface $F$. Equivalently, the original Seiberg-Witten approach of Mrowka-Rollin \cite{MR} applies directly.

Part (b) establishes the relative genus bound without the addition of 1-handles, but the challenge here is that the boundary of the cobordism is disconnected. We use recent work of Matt Hedden \cite{hedden} and apply the hat version of HF to establish an adjunction inequality in the case of disconnected boundary. Alternatively, SW methods apply. 

Part (c) requires no non-vanishing assumptions, and is proved as in parts (a) and (b), depending on the choice of method. There is an extension of the Mrowka-Rollin Seiberg-Witten theory approach which gives the result in part (c), as was pointed out to the author by Matt Hedden. We briefly outline this approach at the end of the proof of part (c).

Theorem \ref{rel-slice-genus} produces an important corollary in the relative setup.

\begin{corollary}\label{cobrelsliceg} 
Let $W$ be an oriented cobordism between the contact 3-manifolds $(M_1,\xi_1)$ and $(M_2,\xi_2)$ so that $\partial W=-M_1\cup M_2$ and let $K_i\subset (M_i,\xi_i)$ be a Legendrian knot for $i=1,2$. 
\begin{enumerate}[(a)]
\item Let $F$ be any embedded surface in $W$ with $\partial F=K_1\cup K_2$ and $W'$ be any 4-manifold obtained from $W$ by attaching enough 1-handles away from $F$ so that $W'$ has connected contact boundary $(M,\xi)$ such that $\xi\rvert_{M_i}=\xi_i$ for $i=1,2$. 

\subitem (i) If there is a $Spin^{\mathbb{C}}$-structure $\s$ on $W'$ with $F^+_{W'\setminus B,\s\rvert_{W'\setminus B}}\big{(}c^+(\xi_1)\otimes c^+(\xi_2)\big{)}\neq0$, where $B$ is an embedded $4$-ball in the interior of $W'$, then there is an isomorphism $h:\s\rvert_{M}\rightarrow \t_{\xi}$ such that
$$\widetilde{reltb}(K_1,K_2,F)+\rvert \widetilde{relr}(K_1,K_2,F,\s\rvert_{W'},h)\rvert\leq-\chi(F).$$

\subitem(ii) If there is a $Spin^{\mathbb{C}}$-structure $\s$ on $W'$ such that $$sw_{(W',\xi)}(\s,h)\neq 0$$ for an isomorphism $h:\s\rvert_M\rightarrow\t_\xi$, $$\widetilde{reltb}(K_1,K_2,F)+\rvert \widetilde{relr}(K_1,K_2,F,\s\rvert_{W'},h)\rvert\leq-\chi(F).$$

\item Let $F$ be any embedded surface in $W$ with $\partial F=K_1\cup K_2$. 

\subitem (i) If there is a $Spin^{\mathbb{C}}$-structure $\s$ on $W$ with $\widehat{F}_{W\setminus B,\s\rvert_{W\setminus B}}\big{(}c(\xi_1)\otimes c(\xi_2)\big{)}\neq0$, where $B$ is an embedded $4$-ball in the interior of $W$, then there are isomorphisms $h_i:\s\rvert_{M_i}\rightarrow \t_{\xi_i}$ for $i=1,2$, such that for any embedded surface $F$ in $W$ with $\partial F=K_1\cup K_2$, 
$$\widetilde{reltb}(K_1,K_2,F)+\rvert \widetilde{relr}(K_1,K_2,F,\s\rvert_{W},h_1,h_2)\rvert\leq-\chi(F).$$

\subitem (ii) If there is a $Spin^{\mathbb{C}}$-structure $\s$ on $W$ with $\s\rvert_{M_i}=\xi_i$ such that $$sw_{(W,\xi_1\cup\xi_2)}(\s,h_1,h_2)\neq 0$$ for isomorphisms $h_i:\s\rvert_{M_i}\rightarrow \t_{\xi_i}, i=1, 2$, then $$\widetilde{reltb}(K_1,K_2,F)+\rvert \widetilde{relr}(K_1,K_2,F,\s\rvert_{W},h_1,h_2)\rvert\leq-\chi(F).$$

\item If $W$ is a symplectic cobordism, then for any embedded surface $F$ with $\partial F=K_1\cup K_2$, 
$$\widetilde{reltb}(K_1,K_2,F)+\rvert \widetilde{relr}(K_1,K_2,F,\s_\omega,h^{\omega}_1,h^{\omega}_2)\rvert\leq-\chi(F).$$
\end{enumerate}
\end{corollary}

\begin{remark} A special case of Corollary \ref{cobrelsliceg} is the Lagrangian concordance of two Legendrian knots in the symplectization of a contact manifold (see \cite{Ch}). Conjecture 7.4 of \cite{Ch} suggests an explicit formula for the difference of the Thurston-Bennequin numbers of two Legendrian knots connected via an immersed Lagrangian cylinder in the symplectization of a contact 3-manifold. This formula is related to the discussion in Remark \ref{remark1}.
\end{remark}

\section{Acknowledgements} 

The author would like to thank Matt Hedden for many helpful discussions and suggestions, in particular, for pointing out a mistake in an earlier draft of the paper, and for sharing Theorem \ref{reladjunction}. Thanks are due to Danny Ruberman and Yanki Lekili for helpful conversations, and especially to Hao Wu and the George Washington University math department for a warm welcome during a brief visit in the summer of 2008. The travel was generously funded by Franklin W. Olin College of Engineering.

\section{Background}

Consider a closed oriented 3-manifold $M$ and a (transversely) oriented 2-plane distribution $\xi$ on $M$ with a (global) 1-form $\alpha$ on $M$ such that $\xi=\ker\alpha$. In this case, $\xi$ is called a {\em contact structure} on $M$ and $\xi$ is {\em positive} if $\alpha\wedge d\alpha>0$. A knot $K\subset (M,\xi)$ is {\em Legendrian} if it is everywhere tangent to $\xi$. Given a Seifert surface $\Sigma\subset M$ for $K$, we can define two (classical) invariants of $K$, the {\em Thurston-Bennequin number} and the {\em rotation number}. The Thurston-Bennequin number is given by
$$tb_\Sigma(K):=tw_K(\xi,Fr_\Sigma)$$
where $tw_K(\xi,Fr_\Sigma)$ counts the number of $2\pi$-twists of the contact framing relative to the Seifert framing along $K$. Equivalently, $tb_\Sigma (K)$ can be defined as $lk(K',K)=K'\cdot\Sigma$ for an (oriented) push-off $K'$ of $K$ in the direction normal to the contact planes. The Thurston-Bennequin number depends on $\Sigma$ only through its relative homology class $[\Sigma]\in H_2(M;\Z)$. The rotation number is defined to be
$$r_\Sigma(K):=c_1(\xi,u)([\Sigma]),$$
where $u$ is the positive unit vector field along $K$. This computes the relative first Chern class of the trivial $\xi$ over $\Sigma$ (in this case, this is equal to the relative Euler class of $\xi$), restricted to $\partial\Sigma=K$, with the nonzero section given by $u$.

In \cite{MR}, Mrowka and Rollin generalized the definition of the classical invariants of Legendrian knots in the following way. Let $W$ be an oriented 4-manifold with connected boundary $\partial W=M$, where $(M,\xi)$ is a contact 3-manifold, and let $K\subset M$ be a Legendrian knot which is the boundary of an embedded surface $F$ in $W$. Consider a vector field transverse to $\xi$ and extend this vector field to all of $W$ with flow $\{\varphi_t\}$. Then take a push-off $K'=\varphi_\epsilon(K)$ for a small $\epsilon>0$ and let $F'=\varphi_\epsilon(F)$. Since $\partial F$ and $\partial F'$ are disjoint, the intersection number $F'\cdot F$ is well-defined and is called the \textit{Thurston-Bennequin number} of $K$ relative to $F$,
$$tb(K,F):=F'\cdot F.$$
It only depends on $F$ through its relative homology class $[F]\in H_2(W,K;\Z)$. Here, $K$ is oriented as the boundary of the oriented $F$, similarly for $K'$ and $F'$. In the case when $F\subset M$, $tb(K,F)$ coincides with the classical definition $tb_F(K)=K'\cdot F=lk(K',K)=tw_{K}(\xi,Fr_F)$ (for a given homology class $[F]\in H_2(W,K;\Z)$). Now, for $\s\in Spin^{\mathbb{C}}(W)$ there exists an isomorphism $h:\s\rvert_{M}\rightarrow \t_\xi$, where $\t_\xi$ is the canonical $Spin^{\mathbb{C}}$-structure on $M$ associated to $\xi$. With the choice of a complex structure on $\xi$, the determinant line bundle $\det(\t_\xi)$ is canonically isomorpic to $\xi$ as a complex line bundle over $M$, and $h$ induces an isomorphism $\det(h):\det(\s)\rvert_{M}\rightarrow\xi$ (see \cite {KM} or Ch.6 in \cite{OSt}). Let $u$ be the positive unit tangent vector field along $K$. Since $K$ is Legendrian, $u$ gives us a nonzero section of $\xi$ and trivializes it over $K$, therefore, we get a nonzero section of $\det(\s)\rvert_M$. Then define \textit{the rotation number} of $K$ relative to $F$ to be 
$$r(K,F,\s,h):=c_1(\det(\s),\det(h)^{-1}(u))([F]),$$
where $c_1(\det(\s),\det(h)^{-1}(u))$ is the first Chern class of $\det(\s)$ relative to the trivialization induced by $\det(h)^{-1}(u)$. Note that $r(K,F,\s,h)$ depends on $F$ only through its relative homology class $[F]\in H_2(W,K;\Z)$ and on the isomorphism type of the pair $(\s,h)$ in $Spin^{\mathbb{C}}(W,\xi)$ (see \cite{KM}). In the special case when $F\subset (M,\xi)$, $r$ is independent of $(\s,h)$ and gives us the classical definition of the rotation number $r_F(K)=c_1(\xi,u)([F])$, where $c_1(\xi,u)$ is the relative Chern class with respect to the trivialization of $\xi$ along $K$ induced by $u$. In the case when $W$ is a symplectic manifold with symplectic form $\omega$, such that $(W,\omega)$ is a weak symplectic filling of $(M,\xi)$, that is, $\omega\rvert_\xi>0$, the symplectic form determines a canonical $Spin^{\mathbb{C}}$-structure $\s_\omega$ on $W$ and a canonical (up to homotopy) isomorphism $h_\omega:\s_\omega\rvert_M\rightarrow\t_\xi$. In this case, we have for the rotation number of $K$
$$r(K,F,\omega):=r(K,F,\s_\omega,h_\omega).$$

Mrowka-Rollin proved the following theorem.

\begin{theorem}[\cite{MR}]\label{MRslice-genus}
Let $(M,\xi)$ be a 3-dimensional closed contact manifold and $W$ be a compact 4-dimensional manifold with boundary $M$. Suppose we have a Legendrian knot $K\subset Y$, and $\Sigma\subset W$ a connected orientable compact surface with boundary $\partial\Sigma=K$. Then for every relative $Spin^{\mathbb{C}}$-structure $(\s,h)$ with $sw_{(M,\xi)}(\s,h)\neq 0$, we have $$\chi(\Sigma)+tb(K,\sigma)+\rvert r(K,\sigma,\s,h)\rvert\leq 0,$$ where $\chi$ denotes the Euler characteristic.
\end{theorem}

H. Wu proved a reformulation of the Mrowka-Rollin theorem in the language of Heegaard Floer homology.

\begin{theorem}[\cite{Wu1}]\label{slice-genus}
Let $W$ be an oriented 4-manifold with connected boundary $\partial W=M$, $\xi$ a contact structure on $M$, and $K$ a Legendrian knot in $(M,\xi)$.
\begin{enumerate} [(a)] 
\item If there is a $Spin^{\mathbb{C}}$-structure $\s$ on $W$ with $F^+_{W\setminus B,\s\rvert_{W\setminus B}}(c^+(\xi))\neq0$, where $B$ is an embedded 4-ball in the interior of $W$, then there is an isomorphism $h:\s\rvert_M\rightarrow \t_\xi$ such that, for any embedded surface $F$ in $W$ bounded by $K$, $$tb(K,F)+\rvert r(K,F,\s,h)\rvert\leq-\chi(F).$$
\item If $(W,\omega)$ is a weak symplectic filling of $(M,\xi)$, then for any embedded surface $F$ in $W$ bounded by $K$, $$tb(K,F)+\rvert r(K,F,\mathfrak{s}_\omega,h_\omega)\rvert\leq-\chi(F).$$
\end{enumerate}
\end{theorem}

\section{Legendrian Surgery and Symplectic 2-handles}

Weinstein gives a standard symplectic 2-handle model for Legendrian surgery along a Legendrian knot $K$ in a contact 3-manifold (see \cite{Ga1, We}). In particular, if $(W,\omega)$ is a symplectic 4-manifold with contact boundary $(M,\xi)$ and $\omega\rvert_\xi>0$, then performing Legendrian surgery on $K$ gives us a 4-manifold $(W',\omega')$ with contact boundary $(M',\xi')$, where $W$ is obtained from $W$ by attaching a standard symplectic 2-handle to $M=\partial W$, so that $\omega$, $\xi$ extend in a canonical way over the 2-handle to $\omega'$, $\xi'$ and $\omega'\rvert_{\xi'}>0$. Moreover, this construction extends a non-vanishing symplectic vector field near the boundary of $W$ to a non-vanishing symplectic vector field over the symplectic 2-handle. In particular, this gives a convex filling $(W',\omega')$ to the manifold $(M',\xi')$. 

Analogous construction exists in the case of Legendrian surgery along a $K$ in $(M,\xi)$ with a concave filling $(W,\omega)$, that is, $\omega\rvert_\xi<0$ (\cite {Ga2,Ga3}). The attached symplectic 2-handle provides analogous extensions of $\omega$, $\xi$, and a non-vanishing symplectic vector field defined near the boundary of $W$. So the new symplectic manifold $(W',\omega')$ is a concave filling of the new 3-manifold $(M',\xi')$ obtained from the Legendrian surgery on $K$.

\begin{remark} 
When $K$ is a slice knot, i.e., when there exists an embedded surface $\Sigma$ in the convex or concave filling $(W,\omega)$ with $\partial\Sigma=K$, a symplectic 2-handle attachment caps off $\Sigma$ with the core of the 2-handle to give an embedded closed surface $\Sigma'\subset(W',\omega')$ satisfying $\chi(\Sigma')=\chi(\Sigma)+1$.
\end{remark}

\section{Concave filling}

Now consider a closed oriented $3$-manifold $M$ (not necessarily connected), and $\xi$ an oriented contact structure compatible with the orientation of $M$. Let $M$ be the oriented boundary of an oriented $4$-manifold of $W$, then recall that a symplectic form $\omega$ on $W$ is weakly compatible with $\xi$ if the restriction $\omega|_{M}$ is positive on the $2$-plane field $\xi$; or equivalently, if $\alpha\wedge \omega|_{M} >0$. The following is proved in \cite{Eliashberg}.

\begin{theorem}\label{thm:Eliash}{\cite{Eliashberg}} Let $M$ be the oriented boundary of a $4$--manifold $W$ and let $\omega$ be a symplectic form on $W$. Suppose there is a contact structure $\xi$ on $M$ that is compatible with the orientation of $M$ and weakly compatible with $\omega$. Then we can embed $W$ in a closed symplectic $4$-manifold $(X,\Omega)$ in such a way that $\Omega|_{W} = \omega$.
\end{theorem}

In \cite{Eliashberg} $X$ is constructed as a smooth manifold as follows. If the
components of $M$ are $M_{1},\dots,M_{n}$, then choose an open-book decomposition of each $M_{i}$ with binding $B_{i}$. These open-book decompositions are compatible with the contact structures $\xi|_{M_{i}}$ in the sense of \cite{Giroux}. Take each binding $B_{i}$ to be connected. Let $W'$ be obtained from $W$ by attaching a $2$-handle along each knot $B_{i}$ with zero framing. The boundary $M' = \partial W'$ is the union of $3$-manifolds $M_{i}'$, obtained from $M_{i}$ by zero surgery: each $M_{i}'$ fibers over the circle with typical fiber $S_{i}$. The genus of $S_{i}$ is the genus of the leaves of the open-book decomposition of $M_{i}$. For each $i$, one then constructs a symplectic Lefschetz fibration $p_{i} :  Z_{i} \rightarrow B_{i}$ over a $2$--manifold-with-boundary $B_{i}$, with $\partial B_{i} =S^{1}$. One constructs $Z_{i}$ to have
the same fiber $S_{i}$, and $\partial Z_{i} = - M'_{i}$. The $4$-manifold $X$ is obtained as the union of $W'$ and the $Z_{i}$, joined along their common boundaries $M_{i}'$.

\section{Heegaard-Floer Homology} 

\subsection{Heegaard-Floer homology: $\widehat{HF}$ and $HF^+$.} Given a connected closed 3-manifold $M$ with $Spin^\C$-structure $\t$, Ozsv\'ath-Szab\'o defined a collection of Heegaard-Floer groups associated to $(M,\t)$ (\cite{OS4}), which are $\Z[U]$-modules. We will use the groups $\widehat{HF}(M,\t)$ and $HF^+(M,\t)$. For a $\Z[H^1(M)]$-module $\M$, the $\M$-twisted Heegaard-Floer homologies $\underline{HF}^+(M,\t;\M)$ and $\underline{\widehat{HF}}(M,\t;\Z)$ are $\Z[U]\otimes\Z[H^1(M)]$-modules (see \cite{OS5}). A $\Z[H^1(M)]$-module homomorphism $\theta:\M_1\rightarrow\M_2$ induces homomorphisms $\widehat{\Theta}:\underline{\widehat{HF}}(M,\t;\M_1)\rightarrow\underline{\widehat{HF}}(M,\t;\M_2)$ and $\Theta^+:\underline{HF}^+(M,\t;\M_1)\rightarrow\underline{HF}^+(M,\t;\M_2)$ (see \cite{OS5}). We can consider $\Z$ as a $\Z[H^1(M)]$ module and obtain the Heegaard-Floer homology $\widehat{HF}(M,\t)$ (resp., $HF^+(M,\t)$) defined with an appropriate orientation system as a special case of the twisted Heegaard Floer homology $\underline{\widehat{HF}}(M,\t;\Z)$ (resp., $\underline{HF}^+(M,\t;\Z)$) (see \cite{OS2,OS3}).

Moreover, for $[\omega]\in H^2(M;\R)$, Ozsv\'ath-Szab\'o defined the Heegaard-Floer homologies $\underline{\widehat{HF}}(M,\t;[\omega])$ and $\underline{HF}^+(M,\t;[\omega])$ twisted by $[\omega]$. Consider the polynomial ring 
$$\Z[\R]=\Bigg{\{}\sum_{i=1}^kc_iT^{s_i}\rvert k\in\Z_{\geq 0},c_i\in\Z,s_i\in\R\Bigg{\}}.$$
Take a cohomology class $[\omega]\in H^2(M;\R)$. For $[\nu]\in H^1(M)$, there is an action $e^{[\nu]}\cdot T^{s}=T^{s+\int_M\nu\wedge\omega}$ which gives $\Z[\R]$ a $\Z[H^1(M)]$-module structure. We denote this module by $\Z[\R]_{[\omega]}$ and the Heegaard-Floer homologies of $M$ twisted by $[\omega]$ by $\underline{\widehat{HF}}(M,\t;\Z[\R]_{[\omega]})$ and  $\underline{HF}^+(M,\t;\Z[\R]_{[\omega]})$. In the case of a disjoint collection $(M_1,\t_1),\dots,(M_n,\t_n)$ of 3-manifolds with a $Spin^\C$-structure, choosing a cohomology class $[\omega_i]\in H^2(M_i;\R)$ on each $M_i$, we get a cohomology class $[\omega_1]\oplus\cdots\oplus[\omega_n]\in H^2(M_1,\R)\oplus\cdots\oplus H^2(M_n,\R)$ and a generalized action given by 
$$e^{[\nu_1]\oplus\cdots\oplus[\nu_n]}\cdot T^s=T^{s+\int_{M_1\cup\cdots\cup M_n}\sum_{i,j}\nu_i\wedge\omega_j}$$
for $[\nu_i]\in H_1(M_i;\R)$ can be defined on 

$$\underline{\widehat{HF}}(M_1,\t_1)\otimes\cdots\otimes\underline{\widehat{HF}}(M_n,\t_n).$$
In a suitable context, we can obtain a twisted Heegaard-Floer homology product given by 
$$\underline{\widehat{HF}}(M_1\cup\cdots\cup M_n,\t_1\otimes\cdots\otimes\t_n;\Z[\R]_{[\omega_1\oplus\cdots\oplus\omega_n]}).$$ 

\subsection{Adjunction Inequalities from Heegaard Floer}

H. Wu \cite{Wu1} observes that the following theorem of Ozsv\'ath-Szab\'o applies in the twisted case.

\begin{theorem}\cite{OS}\label{adjunction} 
Let $\Sigma\hookrightarrow M$ be a closed oriented surface with $g(\Sigma)\geq1$ and $\t$ a $Spin^\C$-structure on $M$ with $\underline{HF}^+(M,\t;\M)\neq 0$, then 
$$[\Sigma]^2+|\langle c_1(\t),[\Sigma]\rangle|\leq-\chi(\Sigma),$$
in particular, 
$$|\langle c_1(\t),[\Sigma]\rangle|\leq-\chi(\Sigma).$$
\end{theorem}

Recently, M. Hedden has extended the adjunction inequality for closed surfaces to surfaces with boundary. In particular, he has proved the following theorem.

\begin{theorem}\cite{hedden}\label{reladjunction}
Let $W$ be a smooth four-manifold with $\partial W=-Y_2\cup Y_1$, and let $K\subset Y_2$ be a null-homologous knot. Suppose there exists $\t\in Spin^\C(W)$ such that
$$\alpha\in Im\Big{(} \widehat{F}_{W,\t}:\widehat{HF}(Y_1)\rightarrow\widehat{HF}(Y_2)\Big{)}.$$
Then
$$\rvert c_1(\t)([\Sigma]\rvert+[\Sigma]^2+2\tau_\alpha(Y_2,K)\leq 2g(\Sigma),$$
where $\iota:(\Sigma,\partial\Sigma)\subset (W,Y_2)$ is any smoothly, properly embedded surface with $\iota\rvert_{\partial\Sigma}\cong K$.
\end{theorem}

Here, $\tau_\alpha(Y_2,K)$ denotes, in vague terms, a lower bound on the values of a knot invariant that has the form of a filtration of the chain complex $\widehat{CF}(Y,\s)$ for $\widehat{HF}(Y,\s)$. For a precise definition, consider \cite{hedden}. We are interested in the following particular corollary, in which the term $\tau_\alpha(Y_2,K)$ vanishes.

\begin{corollary}\cite{hedden}\label{adjunctionhat}
In the setup of Theorem \ref{reladjunction}, consider $K$ to be the unknot in $Y_2$. Then for any closed surface $\Sigma$ in a smooth four-manifold $W'$, consider $\Sigma\setminus D$, where $D$ is a 2-disc in $D^4$ with $K=\partial D$ and $W=W'\setminus D^4$ (take the closure), we have $2\tau_\alpha(Y_2,K)=0$
$$\rvert c_1(\t)([\Sigma\setminus D])\rvert+[\Sigma\setminus D]^2\leq 2g(\Sigma\setminus D),$$
which implies
$$\rvert c_1(\t)([\Sigma])\rvert+[\Sigma]^2\leq 2g(\Sigma),$$
since $c_1(\t)([\Sigma])=c_1(\t)([\Sigma\setminus D])$, $[\Sigma]^2=[\Sigma\setminus D]^2$, and $g(\Sigma)=g(\Sigma\setminus D)$.
\end{corollary}

\subsection{Homomorphisms and the Contact Invariant.}

For a cobordism $W$ from a 3-manifold $M_1$ to another 3-manifold $M_2$, and a $Spin^C$-structure on $W$, we have homomorphisms $\widehat{F}_{W,\s}:\widehat{HF}(M_1,\s\rvert_{M_1})\rightarrow \widehat{HF}(M_2,\s\rvert_{M_2})$ and $F^+_{W,\s}:HF^+(M_1,\s\rvert_{M_1})\rightarrow HF^+(M_2,\s\rvert_{M_2})$. For a $\Z[H^1(M)]$-module $\M$ let $\M(W)=\M\otimes_{\Z[H^1(M_1)]}\Z[\delta H^1(\partial W)]$, where $\delta:H^1(\partial W)\rightarrow H^2(W,\partial W)$ is the map in the long exact sequence of the pair $(W,\partial W)$. Then $\s$ and $W$ induce a homomorphisms (see \cite{OS1})
$$\underline{\widehat{F}}_{W,\s}:\underline{\widehat{HF}}(M_1,\s\rvert_{M_1};\M)\rightarrow\underline{\widehat{HF}}(M_2,\s\rvert_{M_2};\M(W)),$$
$$\underline{F}^+_{W,\s}:\underline{HF}^+(M_1,\s\rvert_{M_1};\M)\rightarrow\underline{HF}^+(M_2,\s\rvert_{M_2};\M(W)).$$ Each of these homomorphisms represents an equivalence class $[\underline{\widehat{F}}_{W,\s}]$ (resp., $[\underline{F}^+_{W,\s}]$) which is well-defined for a given $(W,\s)$. Similarly, for $[\omega]\in H^2(M;\R)$, we get an equivalence class $[\underline{\widehat{F}}_{W,\s;[\omega]}]$ (resp., $[\underline{F}^+_{W,\s;[\omega]}]$) where 
$$\underline{\widehat{F}}_{W,\s}:\underline{\widehat{HF}}(M_1,\s\rvert_{M_1};[\omega\rvert_{M_1}])\rightarrow\underline{\widehat{HF}}(M_2,\s\rvert_{M_2};[\omega\rvert_{M_2}]),$$
$$\underline{F}^+_{W,\s}:\underline{HF}^+(M_1,\s\rvert_{M_1};[\omega\rvert_{M_1}])\rightarrow\underline{HF}^+(M_2,\s\rvert_{M_2};[\omega\rvert_{M_2}]).$$ These homomorphisms have the following composition properties. 

Let $F^\circ,HF^\circ$ stand for $\widehat{F},\widehat{HF}$ and $F^+,HF^+$ for the rest of this section.

\begin{theorem}\cite{OS1}\label{composition} 
Let $W_1,W_2$ be cobordisms connecting the 3-manifolds $M_1$, $M_2$, and $M_3$ with $W=W_1\cup_{M_2}W_2$. Then:
\begin{enumerate}
\item[(a)] For $Spin^\C$-structures $\s_i\in Spin^\C(W_i), i=1,2$, $$F^\circ_{W_2,\s_2}\circ F^\circ_{W_1,\s_1}=\sum_{\{\s\in Spin^\C(W)\ \rvert\ \s\rvert_{W_i}=\s_i\}}\pm F^\circ_{W,\s}.$$
\item[(b)] For $\s\in Spin^\C(W)$, $\s_i=\s\rvert_{W_i}$, and $\Z[H^1(M_1)]$-module $\M_i$, there exists $\underline{F}^\circ_{W_i,\s_i}\in [\underline{F}^\circ_{W_i,\s_i}]$, $i=1,2$, such that $[\underline{F}^\circ_{W,\s}]=[\Pi\circ\underline{F}^\circ_{W_1,\s_1}\circ\underline{F}^\circ_{W_2,\s_2}]$, where $\Pi:\M(W_1)(W_2)\rightarrow\M(W)$ is the natural homomorphism.
\end{enumerate}
\end{theorem}

\begin{remark} Theorem \ref{composition} holds also in the case when $\partial W$ is disconnected for the maps on $\widehat{HF}$ induced by counting polygons.
\end{remark}

We also have the following composition property under blow-up.

\begin{theorem} \cite{OS1}\label{blow-up} 
Let $W$ be a cobordism from a 3-manifold $M_1$ to a 3-manifold $M_2$, $\s$ a $Spin^\C$-structure on $W$. Blow up an interior point of $W$ to obtain another cobordism $\widehat{W}$ from $M_1$ to $M_2$, let $\widehat{\s}$ be the lift of $\s$ to $\widehat{W}$ with $\langle c_1(\widehat{\s}),[E]\rangle=-1$, where $E$ is the exceptional sphere. Then $F^\circ_{W,\s}=F^\circ_{\widehat{W},\widehat{\s}}$.
\end{theorem}

\begin{remark} Theorem \ref{blow-up} holds also in the case when $\partial W$ is disconnected for the maps on $\widehat{HF}$ induced by counting polygons.
\end{remark}

The Ozsv\'ath-Szab\'o invariant $c(\xi)$ of a contact 3-manifold $(M,\xi)$ is an element of $\widehat{HF}(-M,\t_\xi)/\{\pm1\}$, where $\t_\xi$ is the $Spin^\mathbb{C}$-structure associated  with $\xi$. Then $c^+(\xi)\in HF^+(-M,\t_\xi)/\{\pm 1\}$ is the image of $c(\xi)$ under the map $\widehat{HF}(-M)\rightarrow HF^+(-M)$ (\cite{OS}). For a $\Z(H^1(M)]$-module $\M$, we have  the following twisted versions of the Ozsv\'ath-Szab\'o contact invariants: $c(\xi;\M)\in \widehat{HF}(-M,\t_\xi;\M)/\Z[H^1(M)]^\times$ and $c^+(\xi;\M)\in HF^+(-M,\t_\xi;\M)/\Z[H^1(M)]^\times$, where $\Z[H^1(M)]^\times$ denotes the set of units in $\Z[H^1(M)]$. Then for $[\omega]\in H^2(M;\R)$, we have the $[\omega]$-twisted invariants $c(\xi;[\omega])\in \widehat{HF}(-M,\t_\xi;[\omega])/\{\pm T^s\ \rvert\ s\in\R\}$ and $c^+(\xi;[\omega])\in HF^+(-M,\t_\xi;[\omega])/\{\pm T^s\ \rvert\ s\in\R\}$ (see \cite{OS1}). Note that these contact invariants are nonzero for tight $\xi$.

We have the following theorems describing the behavior of the contact invariants under cobordisms induced by Legendrian surgeries.

\begin{proposition}\cite{OS}\label{legsurgery}
If $(M',\xi')$ is the result of Legendrian surgery on a Legendrian link in $(M,\xi)$, then $\widehat{F}_{W,\s_\omega}(c(\xi'))=c(\xi)$, where $W$ is the cobordism induced by the surgery and $\s_\omega$ is the canonical $Spin^\C$-structure on $W$ associated to $\omega$. Moreover, P. Ghiggini (\cite{Gh1}) observes that the argument extends to show that $\widehat{F}_{W,s}(c(\xi'))=0$ for any $Spin^\C$-structure $\s\ncong\s_\omega$ on $W$.
\end{proposition}

\begin{remark} Proposition \ref{legsurgery} holds also in the case when $\partial W$ is disconnected for the maps on $\widehat{HF}$ induced by counting polygons.
\end{remark}

\begin{proposition}\cite{Gh1}\label{invariant}
If $(M',\xi')$ is the result of Legendrian surgery on a Legendrian link in $(M,\xi)$, then $F^+_{W,\s_\omega}(c^+(\xi'))=c^+(\xi)$, where $W$ is the cobordism induced by the surgery and $\s_\omega$ is the canonical $Spin^\C$-structure on $W$ associated to $\omega$. Moreover, $F^+_{W,s}(c^+(\xi'))=0$ for any $Spin^\C$-structure $\s\ncong\s_\omega$ on $W$.
\end{proposition}

We also have the following non-vanishing result in the case of a symplectic cobordism from a contact 3-manifold to the standard tight contact 3-sphere as the boundary of a 4-ball.

\begin{theorem} \cite{OS1}\label{non-zero}
Let $(M,\xi)$ be a contact 3-manifold with a weak symplectic filling $(W,\omega)$. Let $B$ be an embedded 4-ball in the interior of $W$, so $W\setminus B$ is considered as a cobordism from $-M$ to $-\partial B$. Then $\underline{\widehat{F}}_{W\setminus B,\s_\omega\rvert_{W\setminus B};[\omega\rvert_{W\setminus B}]}(c(\xi;[\omega\rvert_M]))\neq 0$ and $\underline{F}^+_{W\setminus B,\s_\omega\rvert_{W\setminus B};[\omega\rvert_{W\setminus B}]}(c^+(\xi;[\omega\rvert_M]))\neq 0$, where $\s_\omega$ is the $Spin^\C$-structure on $W$ associated to $\omega$.
\end{theorem}

\section{Mrowka-Rollin Seiberg-Witten Invariants and Adjunction Inequality}

In \cite{KM}, Kronheimer and Mrowka defined the Sieberg-Witten invariant of a connected oriented smooth four-manifold $W$ whose boundary $M$ carries a contact structure $\xi$. The map is defined by $sw_{(W,\s)}:(Spin^\C,H)\rightarrow\Z$, where $H$ is the group of isomorphisms $h:\s\rvert_M\rightarrow\s_\xi$ of the restriction of an element $\s\in Spin^\C$ to $M$ and the canonical $Spin^\C$-structure $\s_\xi$ on $M$ induced by the contact structure.

The construction for the proof of Theorem \ref{rel-slice-genus} is really based on the original construction in \cite{KM} and the fact that the Seiberg-Witten invariant defined by Kronheimer-Mrowka behaves well under Weinstein 2-handle attachments. The key result that is applied here is the adjunction inequality.

\begin{proposition}[Adjunction inequality]\label{KMr} Let $W$ be a 4-manifold with contact boundary $(M,\xi)$ and an element $(\s,h)\in Spin^\C(W,\xi)$ such that $$sw_{(W,\xi)}(\s,h)\neq 0.$$
Then every closed surface $\Sigma\subset W$ with $[\Sigma]^2=0$ and genus at least 1 satisfies $$\rvert c_1(\s)\cdot\Sigma\rvert\leq-\chi(\Sigma).$$

\end{proposition}

In particular, the Seiberg-Witten invariant defined by Kronheimer-Mrowka in \cite{KM} generalizes in the case when $W$ has disconnected boundary $(M_1,\xi_1)\cup\cdots(M_n,\xi_n)$, so we have a version of the adjunction inequality for this case.

\section{Relative Slice Thurston-Bennequin Inequalities}

The following lemma is a generalization from \cite{MR, Wu1}.

\begin{lemma} \cite{MR, Wu1}\label{adjust}
Let $W$ be an oriented 4-manifold with $\partial W=M_1\cup\cdots\cup M_n$, $\xi_i$ a contact structure on $M_i$, and $\s$ a $Spin^\C$-structure on $W$ with isomorphisms $h_i:\s\rvert_{M_i}\rightarrow\t_{\xi_i}$ for $i=1,\dots,n$. Let $K_i\subset (M_i,\xi_i)$ be a Legendrian knot, and $F\subset W$ an embedded surface bounded by $K_1\cup\cdots\cup K_n$. Then there exist Legendrian knots $K_i'\subset (M,_i,\xi_i)$ for $i=1,\dots,n$ and an embedded surface $F'\subset W$ with boundary $K_1'\cup\cdots\cup K_n'$ such that $$tb(K_1',\dots,K_n',F')+\rvert r(K_1',\dots,K_n',F',\s,h_1,\dots,h_n)\rvert+\chi(F')=\hspace{1in}$$
$$\hspace{1in}=tb(K_1,\dots,K_n,F )+\rvert r(K_1,\dots,K_n,F,\s,h_1,\dots,h_n)\rvert+\chi(F),$$ 

$$\chi(F')\leq-n,$$ and 
$$tb(K_1',\dots,K_n',F')\geq n.$$
\end{lemma}

\begin{proof} If we have $\chi(F)\leq-n$, and $tb(K_1,K_2,F)\geq n$, we are done. If not, the proof is identical to the proof in  \cite{Wu1}, with the only point here that we have freedom which of the $K_i\subset (M_1,\xi_1)$ to connect sum with $n$ trefoils $T\subset (S^3,\xi_{std})$ with $tb(T)=1$ with respect to a Seifert surface $\Sigma_T\subset (S^3,\xi_{std})$. We could connect each one with one trefoil or only one with all $n$, etc., and we will still have the above result in each case. We then have a knot $K_i'\subset(M_1',\xi_1')$, where $M_i'$ is diffeomorphic to $M_i$ and $\xi_i'$ is isotopic to $\xi_i$ after we identify $M_i'$ and $M_i$ for $i=1,\dots,n$. Thus we obtain an embedded surface $F'\subset W$ with the desired properties.
\end{proof}

\begin{proof}[{\bf{\em Proof of Theorem \ref{rel-slice-genus}(a)(i)}}] This is a direct application of Theorem \ref{slice-genus}. We briefly recall the construction used in both theorems in order to motivate a generalization.

By Lemma \ref{adjust}, we only need to prove Theorem \ref{rel-slice-genus}(a)(i) for $K_1,\dots,K_n$ and $F$ with $\widetilde{tb}(K_1,\dots,K_n,F)\geq n\ \text{and}\ \chi(F)\leq-n$.
We assume these are true throughout the proof.

We first attach 1-handles between the boundary components $M_i$ of $W$ by identifying 3-balls $B_i^3\subset M_i$ and $B^3_j\subset M_j$ disjoined from the boundary components of $F$ with the 3-balls $B^3\times\{0,1\}$, which are the ``ends" of the 1-handle $B^3\times B^1\cong B^3\times[0,1]$. The gluing is done via an orientation-reversing diffeomorphism so that we obtain an oriented 4-manifold with connected contact oriented boundary $(M,\xi)$ (the contact structures $\xi_i$ are extended via the contact structure on the boundary $\partial B^3\times[0,1]$). In particular, for a $Spin^{\C}$-structure $\s$ with $\s\rvert_{M_i}=\t_{\xi_i}$, we have a unique extension across the (symplectic) 1-handles. We let the new 4-manifold with boundary $(M,\xi)$ be called $W'$.

Then for each $i=1,\dots,n$, we perform Legendrian surgery along $K_i$ which gives us a symplectic cobordism $(V_i,\omega_i')$ from $(M_i,\xi_i)$ to a contact $3$-manifold $(M_i',\xi_i')$ (c.f. \cite{We,Wu1}). In order to apply Theorem \ref{invariant}, we will consider these as a surgery on the link $K_1\cup\cdots\cup K_n$ in the connected contact 3-manifold $(M,\xi)$.  Let $(V,\omega)$ denote the union of these cobordisms from $(M,\xi)$ to $(M'\xi')$. By Proposition \ref{non-zero}, $F^+_{V,\s_\omega}(c^+(\xi'))=c^+(\xi)$. Let $\widetilde{W}=W'\cup_MV$. Then by Theorem \ref{composition},
$$ \sum_{\begin{matrix}\{\widetilde{\s}\in Spin^{\C}(\widetilde{W})~|~\\ \widetilde{\s}|_{W'}\cong\s, \widetilde{\s}|_V\cong\s_\omega\}\end{matrix}}\pm F^+_{\widetilde{W},\widetilde{\s}}(c^+(\xi'))=F^+_{W',\s}\circ F^+_{V,\s_\omega}(c^+(\xi'))=F^+_{W',\s}(c^+(\xi))\neq0.$$
 
So there is a $\widetilde{\s}\in Spin^\C(\widetilde{W})$ with $\widetilde{\s}\rvert_{W'}\cong\s$, $\widetilde{\s}\rvert_V\cong\s_\omega$, and $F^+_{\widetilde{W},\widetilde{\s}\rvert_{\widetilde{W}}}(c^+(\xi'))\neq 0$.

Denote the above isomorphisms $g:\widetilde{\s}\rvert_{W}\cong\s$, $k:\widetilde{s}\rvert_V\cong\s_\omega$, and the natural isomorphism $f:\s_\omega\rvert_{M'}\rightarrow\t_{\xi'}$. Define $h:\s\rvert_{M'}\rightarrow\t_{\xi'}$ by $h=f\circ k\circ g^{-1}$.

Let $\widetilde{F}\subset \widetilde{W}$ be $F$ capped off by the cores of the 2-handles form the Legendrian surgery, then $\chi(\widetilde{F})=\chi(F)+n\leq0$, $[\widetilde{F}]\cdot[\widetilde{F}]=tb(K_1,\dots,K_n, F)-n\geq0$, and $c_1(\widetilde{\s})([\widetilde{F}])=r(K_1,\dots,K_n, F, \s\rvert_{W'}, h)$.

Blow-up $tb(K_1,\dots,K_n,F)-n$ points on the core of any of the 2-handles to obtain a new 4-manifold $\widehat{W}$ with a natural projection $\pi:\widehat{W}\rightarrow\widetilde{W}$. Lift $\widetilde{\s}$ to $\widehat{W}$ and call the lift $\widehat{\s}$, choose the particular lift that evaluates -1 on each exceptional sphere coming from $\pi^{-1}(\widetilde{F})$ and let $\widehat{F}$ be the lift of $\widetilde{F}$ to $\widehat{W}$ obtained by removing the exceptional spheres from $\pi^{-1}(\widetilde{F})$. Then $\chi(\widehat{F})=\chi(\widetilde{F})=\chi(F)+n$, $[\widehat{F}]\cdot[\widehat{F}]=0$, and $c_1(\widehat{\s},[\widehat{F}])=r(K_1,\dots,K_n, F, \s\rvert_{W'}, h)+tb(K_1,\dots,K_n,F)-n$.

Since $[\widetilde{F}]\cdot[\widetilde{F}]=0$, $\widehat{F}$ has a neighborhood $U\subset\widehat{W}$ diffeomorphic to $\widehat{F}\times D^2$. Consider an embedded 4-ball $B\subset\widetilde{W}$ and let $\widehat{B}\subset U\subset \widehat{W}$ be the pre-image of $B\subset W\subset\widetilde{W}$ under $\pi$. Then, by Theorem \ref{blow-up}, $F^+_{\widehat{W}\setminus\widehat{B},\widehat{\s}\rvert_{\widehat{W}\setminus\widehat{B}}}\big{(}c^+(\widehat{\xi'})\big{)}=F^+_{\widetilde{W}\setminus B,\widetilde{\s}\rvert_{\widetilde{W}\setminus B}}\big{(}c^+(\xi')\big{)}\neq 0$. Since $F^+_{\widehat{W}\setminus\widehat{B},\widehat{\s}\rvert_{\widehat{W}\setminus\widehat{B}}}$ does not depend on the location of $\widehat{B}$, assume $\widehat{B}\subset (U\setminus\partial U)$. Let $\widehat{W}_1=\widehat{W}\setminus U$ and $\widehat{W}_2=U\setminus \widehat{B}$. We get a composition of the cobordisms $\widehat{W}_1$ with $\partial \widehat{W}_1=-\widehat{M'}\cup-\partial U$ and $\widehat{W}_2$ with $\partial \widehat{W}_2=-\partial U\cup-\partial \widehat{B}$ along $\partial U=\widehat{F}\times S^1$. By Theorem \ref{composition} 
(with $\widehat{\t}_{\xi'}$ the lift of $\t_{\xi'}$ from $\widehat{\s}\rvert_{\widehat{M'}}\rightarrow\widehat{\t}_{\xi'}$), we have the following maps.
$$\underline{F}^+_{\widehat{W}_1,\widehat{\s}\rvert_{\widehat{W}_1}}:HF^+(-\widehat{M'},\widehat{\t}_{\xi'};\Z)\rightarrow\underline{HF}^+(\partial U,\widehat{\s}\rvert_{\partial U};\Z(\widehat{W}_1)),$$
$$\underline{F}^+_{\widehat{W}_2,\widehat{\s}\rvert_{\widehat{W}_2}}:\underline{HF}^+\big{(}-\partial U,\widehat{\s}\rvert_{\partial U};\Z(\widehat{W}_1)\big{)}\
\rightarrow\underline{HF}^+\big{(}-\partial\widehat{B},\widehat{s}\rvert_{\partial\widehat{B}};\Z(\widehat{W}_1)(\widehat{W}_2)\big{)},$$
for $F^+_{\widehat{W}\setminus\widehat{B},\widehat{\s}\rvert_{\widehat{W}\setminus\widehat{B}}}=\Theta\circ\underline{F}^+_{\widehat{W}_2,\widehat{\s}\rvert_{\widehat{W}_2}}\circ\underline{F}^+_{\widehat{W}_1,\widehat{\s}\rvert_{\widehat{W}_1}}$ and $\Theta:\underline{HF}^+\big{(}-\partial\widehat{B},\widehat{\s}\rvert_{\partial\widehat{B}};\Z(\widehat{W}_1)(\widehat{W}_2)\big{)}\rightarrow HF^+\big{(}-\partial\widehat{B},\widehat{\s}\rvert_{\partial\widehat{B}};\Z\big{)}$ induced by the natural projection $\theta:\Z(\widehat{W}_1)(\widehat{W}_2)\rightarrow\Z$. Since 
$F^+_{\widehat{W}\setminus\widehat{B},\widehat{\s}\rvert_{\widehat{W}\setminus\widehat{B}}}\big{(}c^+(\widehat{\xi'})\big{)}\neq 0$, this implies that $\underline{HF}^+\Big{(}-\partial U,\widehat{\s}\rvert_{\partial U};\Z(\widehat{W}_1)\Big{)}\neq 0$, where $\partial U\cong\widehat{F}\times S^1$. Then by Theorem \ref{adjunction}, we have $\langle c_1(\widehat{\s}), [\widehat{F}]\rangle \leq-\chi(\widehat{F})$, so $\widetilde{tb}(K_1,\dots,K_n,F)+\widetilde{r}(K_1,\dots,K_n,F,\s\rvert_{W'},h)\leq-\chi(F)$.

By Remark \ref{remark1}, this construction with reversed orientations of $K_i$ and $F$ yields $\widetilde{tb}(K_1,\dots,K_n,F)-\widetilde{r}(K_1,\dots,K_n,F,\s\rvert_{W'},h)\leq-\chi(F)$,
therefore we obtain the generalized slice Thurston-Bennequin inequality
$$\widetilde{tb}(K_1,\dots,K_n,F)+\rvert \widetilde{r}(K_1,\dots,K_n,F,\s\rvert_{W'},h)\rvert\leq-\chi(F).$$ \end{proof}

\begin{proof}[{\bf{\em Proof of Theorem \ref{rel-slice-genus}(a)(ii)}}] This is a direct application of Theorem \ref{slice-genus} of Mrowka-Rollin, which uses the same topological setup as Theorem (a)(i) above.
\end{proof}

\begin{proof}[{\bf{\em Proof of Theorem \ref{rel-slice-genus}(b)(i)}}] We follow the arguments in the proof of part (a), but we want to avoid the preliminary attaching of 1-handles to connect the boundary of the 4-manifold $W$. This requires us to use the ``hat" version of Heegaard-Floer homology which is better behaved. An important point here is that although the Heegaard-Floer Homology groups are not defined for disconnected 3-manifolds, we can make sense of the maps as an extension of the usual maps to counting polygons in $W$. The goal here is to still obtain a genus bound on the surface $F$ by capping it off and applying an appropriate version of the adjunction inequality.

First, we perform Legendrian surgery along each boundary component $K_i$ of $F$, which gives cobordisms $(V_i,\omega_i)$ from $(M_i,\xi_i)$ to $(M_i',\xi_i')$ given in the standard way by taking a canonical product neighborhood of $M_i$ and attaching a Weinstein 2-handle along one end. By Proposition \ref{non-zero}, $F^+_{V_i,\s_{\omega_i}}(c^+(\xi_i'))=c^+(\xi_i)$. Let $\widetilde{W}=\bigcup_iW\cup_{M_i}V_i$. Then by Theorem \ref{composition},
 
$$ \sum_{\{\widetilde{\s}\in Spin^{\C}(\widetilde{W}) ~|~\widetilde{\s}|_{W}\cong\s, ~\widetilde{\s}|_{V_i}\cong\s_{\omega_i}\}}\pm F^+_{\widetilde{W},\widetilde{\s}}\Big{(}c^+(\xi_1')\otimes\cdots\otimes c^+(\xi_n')\Big{)}=\hspace{1in}$$
$$\hspace{1in}=F^+_{W,\s} \Bigg{(} \bigotimes_{i=1}^n F^+_{V_i,\s_{\omega_i}}(c^+(\xi_i'))\Bigg{)}=F^+_{W,\s}\Big{(}c^+(\xi_1)\otimes\cdots\otimes c^+(\xi_n)\Big{)}\neq0.$$ 

Thus, there exists a $\widetilde{\s}\in Spin^\C(\widetilde{W})$ with $\widetilde{\s}\rvert_{W_1}\cong\s_1$ and $\widetilde{\s}\rvert_{V_2}\cong\s_{\omega_2}$ such that

$$F^+_{\widetilde{W},\widetilde{\s}\rvert_{\widetilde{W}}}\Big{(}c^+(\xi_1')\otimes\cdots\otimes c^+(\xi_n')\Big{)}\neq 0.$$

Denote the above isomorphisms $g:\widetilde{\s}\rvert_{W}\cong\s$, $k_i:\widetilde{s}\rvert_{V_i}\cong\s_{\omega_i}$, and the natural isomorphisms $f_i:\s_{\omega_i}\rvert_{M_i'}\rightarrow\t_{\xi_i'}$ for $i=1,\dots,n$. Define $h_i:\s\rvert_{M_i'}\rightarrow\t_{\xi_i'}$ by 
$$h_i=f_i\circ k_i\circ g^{-1}.$$

Capping off $F$ in $\widetilde{W}$ by the cores of the 2-handles from the Legendrian surgeries, we obtain an embedded closed surface $\widetilde{F}$ in $\widetilde{W}$ satisfying $\chi(\widetilde{F})=\chi(F)+n\leq0$ (we capped off by $n$ 2-discs). Then 
$$[\widetilde{F}]\cdot[\widetilde{F}]=tb(K_1,\dots,K_n, F)-n\geq0$$ 
and 
$$c_1(\tilde{\s},[\widetilde{F}])=r(K_1,\dots,K_n, F, \s\rvert_W, h_1,\dots,h_n).$$

Blow-up $tb(K_1,\dots,K_n,F)-n$ points on the core of one or more (or all) of the 2-handle to obtain a new 4-manifold $\widehat{W}$ with a natural projection $\pi:\widehat{W}\rightarrow\widetilde{W}$. Lift $\tilde{\mathfrak{s}}$ to $\widehat{W}$ and call the lift $\hat{\mathfrak{s}}$, choose the particular lift that evaluates -1 on each exceptional sphere coming from $\pi^{-1}(\widetilde{F})$ and let $\widehat{F}$ be the lift of $\widetilde{F}$ to $\widehat{W}$ obtained by removing the exceptional spheres from $\pi^{-1}(\widetilde{F})$. Then we have 
$$\chi(\widehat{F})=\chi(\widetilde{F})=\chi(F)+n$$ 
and 
$$[\widehat{F}]\cdot[\widehat{F}]=0.$$ 
Additionally, 
$$c_1(\widehat{\s},[\widehat{F}])=r(K_1,\dots,K_n, F, \s\rvert_W, h_1,\dots,h_n)+tb(K_1,\dots,K_n,F)-n.$$

Since $[\widetilde{F}]\cdot[\widetilde{F}]=0$, there is a neighborhood $U$ of $\widehat{F}$ in $\widehat{W}$ diffeomorphic to $\widehat{F}\times D^2$. Consider an embedded a 4-ball $B$ in $W$ and let $\widehat{B}\subset U\subset \widehat{W}$ be the pre-image of $B\subset W\subset\widetilde{W}$ under $\pi$. Then, by Theorem \ref{blow-up}, 

$$F^+_{\widehat{W}\setminus\widehat{B},\widehat{\s}\rvert_{\widehat{W}\setminus\widehat{B}}}\Big{(}c^+(\widehat{\xi_1'})\otimes\cdots\otimes c^+(\widehat{\xi_n'})\Big{)}=F^+_{\widetilde{W}\setminus B,\widetilde{\s}\rvert_{\widetilde{W}\setminus B}}\Big{(}c^+(\xi_1')\otimes\cdots\otimes c^+(\xi_n')\Big{)}\neq 0.$$

Since the location of $\widehat{B}$ does not affect the map $F^+_{\widehat{W}\setminus\widehat{B},\widehat{\s}\rvert_{\widehat{W}\setminus\widehat{B}}}$, we can assume that $\widehat{B}$ is in the interior of $U$. 

Let $\widehat{W}_1=\widehat{W}\setminus U$ and $\widehat{W}_2=U\setminus \widehat{B}$. Then this gives us a composition of two cobordisms with the 3-manifold $\partial U=\widehat{F}\times S^1$ as a ``cut". One of them is $\widehat{W}_1$ with $\partial \widehat{W}_1=\big{(}-\widehat{M'_1}\cup\cdots\cup -\widehat{M'_n}\big{)}\cup-\partial U$ and the other one is $\widehat{W}_2$ with $\partial \widehat{W}_2=-\partial U\cup-\partial \widehat{B}$. By Theorem \ref{composition}, we have the following maps.

$$\underline{F}^+_{\widehat{W},\widehat{\s}\rvert_{\widehat{W}_1}}:HF^+\Big{(}-\widehat{M'_1}\cup\cdots\cup -\widehat{M'_n},\widehat{\t}_{\xi_1'}\otimes\cdots\otimes\widehat{\t}_{\xi_n'};\Z\Big{)}\rightarrow\underline{HF}^+\Big{(}\partial U,\widehat{\s}\rvert_{\partial U};\Z(\widehat{W}_1)\Big{)},$$
(where $\widehat{\t}_{\xi_i}$ is the lift of $\t_{\xi_i}$ in the isomorphism $\widehat{\s}\rvert_{\widehat{M_i'}}\rightarrow\widehat{\t}_{\xi_i}$),

$$\underline{F}^+_{\widehat{W}_2,\widehat{\s}\rvert_{\widehat{W}_2}}:\underline{HF}^+\Big{(}-\partial U,\widehat{\s}\rvert_{\partial U};\Z(\widehat{W}_1)\Big{)}\
\rightarrow\underline{HF}^+\Big{(}-\partial\widehat{B},\widehat{s}\rvert_{\partial\widehat{B}};\Z(\widehat{W}_1)(\widehat{W}_2)\Big{)},$$
such that 
$$F^+_{\widehat{W}\setminus\widehat{B},\widehat{\s}\rvert_{\widehat{W}\setminus\widehat{B}}}=\Theta\circ\underline{F}^+_{\widehat{W}_2,\widehat{\s}\rvert_{\widehat{W}_2}}\circ\underline{F}^+_{\widehat{W},\widehat{\s}\rvert_{\widehat{W}_1}},$$
where
$$\Theta:\underline{HF}^+\Big{(}-\partial\widehat{B},\widehat{\s}\rvert_{\partial\widehat{B}};\Z(\widehat{W}_1)(\widehat{W}_2)\Big{)}\rightarrow HF^+\Big{(}-\partial\widehat{B},\widehat{\s}\rvert_{\partial\widehat{B}};\Z\Big{)}$$
is induced by the natural projection $\theta:\Z(\widehat{W}_1)(\widehat{W}_2)\rightarrow\Z$.

Since $F^+_{\widehat{W}\setminus\widehat{B},\widehat{\s}\rvert_{\widehat{W}\setminus\widehat{B}}}\Big{(}c^+(\widehat{\xi_1'})\otimes\cdots\otimes c^+(\widehat{\xi_n'})\Big{)}\neq 0$, this implies that 
$$\underline{HF}^+\Big{(}-\partial U,\widehat{\s}\rvert_{\partial U};\Z(\widehat{W}_1)\Big{)}\neq 0,$$
where $\partial U\cong\widehat{F}\times S^1$.

Then by Theorem \ref{adjunction}, we have 

$$\langle c_1(\widehat{\s}), [\widehat{F}]\rangle \leq-\chi(\widehat{F}),$$ 
so
$$\widetilde{tb}(K_1,\dots,K_n,F)+\widetilde{r}(K_1,\dots,K_n,F,\s\rvert_{W},h_1,\dots,h_n)\leq-\chi(F).$$

Using Remark \ref{remark1}, reversing orientations of all $K_i$ and of $F$ and going through this construction yields
$$\widetilde{tb}(K_1,\dots,K_n,F)-\widetilde{r}(K_1,\dots,K_n,F,\s\rvert_{W},h_1,\dots,h_n)\leq-\chi(F),$$
therefore we obtain the generalized slice Thurston-Bennequin inequality
$$\widetilde{tb}(K_1,\dots,K_n,F)+\rvert \widetilde{r}(K_1,\dots,K_n,F,\s\rvert_{W},h_1,\dots,h_n)\rvert\leq-\chi(F).$$

\end{proof}

\begin{proof}[{\bf{\em Proof of Theorem \ref{rel-slice-genus}(b)(ii)}}] This is a direct application of the generalized version of Proposition\ref{KMr} of Mrowka-Rollin for a 4-manifold with disconnected boundary, which uses the same topological setup as Theorem (b)(i) above.
\end{proof}

\begin{remark}[{\bf{\em Idea for a Proof of Theorem \ref{rel-slice-genus}(c)}}]
If we follow the same argument as in part (b)(i) with symplectic $(W, \omega)$ and perform Legendrian surgery along each $K_i$ we obtain for each $i=1,\dots,n$ new contact 3-manifolds $(M_i',\xi_i')$ containing the boundary of the attached 2-handles, and symplectic cobordisms $(V_i,\omega_i)$ from $(M_i,\xi_i)$ to $(M_i',\xi_i')$. Then let $\widetilde{W}$ be defined as above and note that it is a symplectic manifold with a symplectic form $\widetilde{\omega}$ such that $\widetilde{\omega}\rvert_W=\omega$ and $\widetilde{\omega}\rvert_{V_i}=\omega_i$. Define $\widetilde{F}$ as above by capping $F$ with the cores of the 2-handles, and then blow-up $\widetilde{tb}(K_1,\dots,K_n,F)-n$ points on the cores of some of the 2-handles to obtain $\widehat{W}$ and $\widehat{F}$ as above. Let $\widehat{\omega}$ denote the blown-up symplectic form on $\widehat{W}$. Let $\widehat{\s}$ be the canonical $Spin^\C$-structure associated to $\widehat{\omega}$. As above, we have
$$\chi(\widehat{F})=\chi(F)+n,\ [\widehat{F}]\cdot[\widehat{F}]=0,$$
and
$$\langle c_1(\widehat{\s}),[\widehat{F}]\rangle=\widetilde{tb}(K_1,\dots,K_n,F)+\widetilde{r}(K_1,\dots,K_n,F,\s_\omega,h^{\omega}_1,\dots,h^{\omega}_n)-n.$$

By Theorem \ref{non-zero} with $\widehat{W}\setminus\widehat{B}$ a symplectic cobordism from $-\widehat{M'_1}\cup\cdots\cup-\widehat{M'_n}$ to $-\partial\widehat{B}$, where $\widehat{B}\subset U$, we have

$$\underline{F}^+_{\widehat{W}\setminus \widehat{B},\widehat{\s}\rvert_{\widehat{W}\setminus \widehat{B}};[\widehat{\omega}\rvert_{\widehat{W}\setminus \widehat{B}}]}\Big{(}c^+(\widehat{\xi_1'};[\widehat{\omega}\rvert_{\widehat{M'_1}}])\otimes\cdots\otimes c^+(\widehat{\xi_n'};[\widehat{\omega}\rvert_{\widehat{M'_n}}])\Big{)}\neq 0.$$

Let $\widehat{W_1}=\widehat{W}\setminus U$ and $\widehat{W_2}=U\setminus\widehat{B}$ as above. We have the following maps.

$$\underline{F}^+_{\widehat{W_1}\setminus \widehat{B},\widehat{\s}\rvert_{\widehat{W_1}\setminus \widehat{B}};[\widehat{\omega}\rvert_{\widehat{W_1}\setminus \widehat{B}}]}:HF^+\Big{(}-\widehat{M'_1}\cup\cdots\cup -\widehat{M'_n},\widehat{\t}_{\xi_1'}\otimes\cdots\otimes\widehat{\t}_{\xi_n'};\Z[\R]_{[\widehat{\omega}\rvert_{\widehat{M'_1}}]\oplus\cdots\oplus[\widehat{\omega}\rvert_{\widehat{M'_n}}]}\Big{)}\rightarrow$$

$$\hspace{2.2in}\rightarrow\underline{HF}^+\Big{(}-\partial U,\widehat{\s}\rvert_{\partial U};\Z[\R]_{[\widehat{\omega}\rvert_{\widehat{M'_1}}]\oplus\cdots\oplus[\widehat{\omega}\rvert_{\widehat{M'_n}}]}(\widehat{W}_1)\Big{)},$$

$$\underline{F}^+_{\widehat{W_2}\setminus \widehat{B},\widehat{\s}\rvert_{\widehat{W_2}\setminus \widehat{B}};[\widehat{\omega}\rvert_{\widehat{W_2}\setminus \widehat{B}}]}:\underline{HF}^+\Big{(}-\partial U,\widehat{\s}\rvert_{\partial U};\Z[\R]_{[\widehat{\omega}\rvert_{\widehat{M'_1}}\oplus\cdots\oplus\widehat{\omega}\rvert_{\widehat{M'_n}}]}(\widehat{W}_1)\Big{)}\rightarrow$$

$$\hspace{1.95in}\rightarrow\underline{HF}^+\Big{(}-\partial\widehat{B},\widehat{s}\rvert_{\partial\widehat{B}};\Z[\R]_{[\widehat{\omega}\rvert_{\widehat{M'_1}}]\oplus\cdots\oplus[\widehat{\omega}\rvert_{\widehat{M'_n}}]}(\widehat{W}_1)(\widehat{W}_2)\Big{)},$$

Although these maps can be defined, it is unclear that they are well-defined as invariants and that the coefficients are well-behaved. Things are complicated by the fact that we have many independent $U$ actions, and we need to be careful about the ring over which we form the tensor product (things can get very infinite). ÊWe can define the above maps by counting holomorphic polygons whose Spin$^\C$-structures are the ones associated to the symplectic form, however, but there do not seem to exist any invariance theorems in this context. ÊThat is, it isn't clear that the polygon counts are Spin$^\C$ 4-manifold invariants.
\end{remark}

\begin{remark}[Second idea for Proof of Theorem \ref{rel-slice-genus}(c)] Performing Legendrian surgery along the boundary components of $F$ and embedding the resulting symplectic 4-manifold with multiple boundary components into a closed symplectic 4-manifold allows the application of Proposition \ref{KMr}.
\end{remark}

\begin{remark} Since $\partial \widehat{U}\cong S^1\times \widehat{F}$, we could directly show $\underline{HF}^+(-\partial U,\widehat{\t};\Z[\R])\neq 0$ in some special cases (see \cite{jabuka}).
\end{remark}

\begin{proof}[{\bf{\em Proof of Corollary \ref{cobrelsliceg}}}] Apply Theorem \ref{rel-slice-genus} to the ``pair-of-pants" cobordism between $M_1\cup-M_2$ and $-\partial B$, where $B$ is a 4-ball in $W$, whose lift $\widehat{B}$ will be as in the proof of Theorem \ref{rel-slice-genus}. We perform Legendrian surgery along $K_1\subset -M_1$ and along $K_2\subset M_2$. This amounts to attaching a concave symplectic Weinstein 2-handle to $M_1$ and a convex symplectic Weinstein 2-handle to $M_2$. We blow up (only) the core of the convex 2-handle $\widetilde{reltb}(K_1,K_2,F)-2$ times. Then proceed as in the proofs of Theorem \ref{rel-slice-genus}.
\end{proof}

\end{document}